\documentclass{article}
\usepackage{authblk}
\usepackage[utf8]{inputenc}
\usepackage{graphicx}
\usepackage{amsmath}
\usepackage{amssymb}
\usepackage{amsthm}
\usepackage{amsrefs}

\title{Generalizing Abundancy Index \\ to Gaussian Integers}
\author{Vrishab Krishna \footnote{Vrishab Krishna, 182 Dream Meadows, Bangalore, Karnataka, India} \footnote{2010 {\it Mathematics Subject Classification} Primary 11R11; Secondary 11N80.} \footnote{{\it Keywords: } Gaussian integers, Abundancy Index, Friendly Numbers, $k$-powerful $t$-perfect Numbers}}
\affil{vrishab@cfrce.org}
\affil{Centre for Fundamental Research and Creative Education}
\date{}

\newtheorem{theorem}{Theorem}[section]
\newtheorem{corollary}{Corollary}[theorem]
\newtheorem{lemma}[theorem]{Lemma}

\newtheorem{definition}{Definition}[section]
\newtheorem{property}[]{Property}

\newcommand{\norm}[1]{\left\| #1 \right\|}

\begin{document}

\maketitle

\begin{abstract}
    Abundancy index refers to the ratio of the sum of the divisors of a number to the number itself. It is a concept of great importance in defining friendly and perfect numbers. Here, we describe a suitable generalization of abundancy index to the ring of Gaussian integers ($\mathbb{Z}[i]$). We first show that this generalization possesses many of the useful properties of the traditional abundancy index in $\mathbb{Z}$. We then investigate $k$-powerful $\tau$-perfect numbers and prove results regarding their existence in $\mathbb{Z}[i]$.
\end{abstract}

\section{Introduction}

The idea of abundancy index has been dated to around 2000 years ago, studied by the great mathematicians of Europe such as Fermat, Mersenne, Legendre, and Euler, and, more recently, by Paul Erdös. It is of great importance in the study of perfect and friendly numbers. In this paper, we seek to find a suitable generalization of a function, the abundancy index $I(\eta)$ for complex numbers, specifically the Gaussian integers ($\mathbb{Z}[i]$), study its properties, and apply it to the problem of the existence of different kinds of perfect numbers.

We first begin with notation. We refer to $\mathbb{N}$ as the set of natural numbers (starting from 1), $\mathbb{Z}$ as the set of integers, and $\mathbb{Z}[i]$ as the ring of Gaussian integers. We also define $\|\eta\| = \eta \Bar{\eta} = a^2 + b^2$ to be the norm of a complex number $\eta = a + bi$ and $|\eta| = \sqrt{\|\eta\|} = \sqrt{a^2 + b^2}$ the absolute value. In $\mathbb{Z}[i]$, we give the following definition in order to differentiate between nearly identical factors of numbers:

\begin{definition}
    We define two Gaussian integers $\alpha, \beta \in \mathbb{Z}[i]$ to be associated, i.e., $\alpha \sim \beta$ if there exists a unit $u \in \mathbb{Z}[i]$ ($|u| = 1$ or $u = \pm 1, \pm i$) such that $\alpha = u\beta$.
\end{definition}

For non-zero complex number $\eta$, let $\arg(\eta)$ represent the angle of $\eta$. It is easy to show that for any Gaussian Integer, we have an associate in the first quadrant. Henceforth, when we refer to primes in $\mathbb{Z}[i]$, we only deal with those in the first quadrant, i.e., $\eta \in \mathbb{Z}[i]$ with $\arg(\eta) \in [0, \frac{\pi}{2})$. From here onward $\pi$ will be used to refer to primes in $\mathbb{Z}[i]$.

With the above definitions of notation in mind, we proceed to discuss the traditional sum of divisors function, $\sigma_k(n)$ (where $n = p_1^{e_1} p_2^{e_2} \dots p_l^{e_l} > 0$ is the prime factorization of $n$), given by:

\begin{align*}
    \sigma_k(n) = \sum\limits_{c|n} c^k = \prod\limits_{p | n} \frac{p_i^{k(e_i+1)}-1}{p^k-1}.
\end{align*}
We write $\sigma_1(n) = \sigma(n)$. From the above definition, $\sigma_k(n)$ is a weakly multiplicative function: $\sigma_k(mn) = \sigma_k(m)\sigma_k(n)$ for $(m, n) = 1$, i.e., $m$ and $n$ are co-prime.

The abundancy index of $n \in \mathbb{N}$ is defined as the following:
\begin{align*}
    I_k(n) = \frac{\sigma_k(n)}{n^k} = \sigma_{-k}(n),
\end{align*}
with $I_1(n) = I(n)$. $I_k(n)$ is referred to as the $k$-powerful abundancy in the rest of the paper (similar to the definition found in \cite{Luo_2018}).

Abundancy index is of great importance in friendly and solitary numbers. Two numbers are said to be `friendly' if they have the same abundancy, i.e.,  $I(n_1) = I(n_2)$ if and only if  $n_1$ and $n_2$ are friendly. For example, all perfect numbers form a set of friendly numbers with abundancy 2. Conversely, a number that has no friends is said to be `solitary'. For example, 1, 18, 48, and so on are solitary numbers as no other numbers share the same abundancy \cite{Hickerson_2002}. Friendly and solitary numbers still constitute a number of unsolved problems in modern number theory. For example, the question whether 10 is solitary remains an unsolved problem. More generally, we have the concept of $k$-powerful friendly numbers where the $k$-powerful abundancy of two numbers are equal, i.e., $I_k(n_1) = I_k(n_2)$ $\iff$ $n_1$ and $n_2$ are $k$-powerful friendly. Conversely, we have $k$-powerful solitary numbers: numbers whose $k$-powerful abundancy is not shared by any other number.

We now discuss the Gaussian integers $\mathbb{Z}[i]$. Forming an integral domain, the Gaussian integers have been used to prove innumerable important results and novel uses have even been found in coding theory \cite{Bouyuklieva_2013}. Due to its importance in various spheres of modern number theory, the ring $\mathbb{Z}[i]$ and its properties is of prime importance. A complex-to-integer generalization of abundancy index is described by Defant \cite{Defant_2015}. However, such an approach loses the important information of factorization in $\mathbb{Z}[i]$ (of the imaginary component). Here, we endeavor to describe a complex-to-complex definition of abundancy.

Perfect numbers have been an important object of mathematical studies for millennia. Numbers whose sum of divisors is equal to twice the number itself are said to be perfect, i.e., they have an abundancy of 2. Close relationships between perfect numbers and Mersenne Primes have been found in the set of integers as well as that of Gaussian integers \cite{McDaniel_1974}. Several generalizations of these numbers are known of which one of the most popular is the multiply-perfect number, defined to be numbers of integral abundancy. Generalizing further, we have $k$-powerful multiply-perfect numbers with  $I_k(n) = t \in \mathbb{N}$. In this paper, we take this concept, extend it to $\mathbb{Z}[i]$: the $k$-powerful $\tau$-perfect numbers.

The paper is structured as follows. We begin by describing a suitable generalization of abundancy to $\mathbb{Z}[i]$. We then proceed to study and prove several important properties of this generalization. Finally, we prove some interesting results regarding the bounds on the function and the existence of $k$-powerful $\tau$-perfect numbers for various $k$.

\section{Extension to the $\mathbb{Z}[i]$}
There are innumerable possible generalizations of abundancy from $\mathbb{Z}$ to $\mathbb{Z}[i]$. We propose the following complex-to-complex definition of abundancy index:

\begin{definition} \label{def_I}
    Given $\eta = u\pi_1^{e_1}\pi_2^{e_2}\dots\pi_k^{e_k}$ is the factorization of $\eta \in \mathbb{Z}[i]$ where $u$ is a unit ($|u| = 1$) and each $\pi$ lies in the first quadrant. We have the sum of divisors function $\sigma_k(\eta)$ (as defined in \cite{Spira_1961}):
    \begin{align*}
        \sigma_k(\eta) = \prod \frac{\pi_i^{k(e_i+1)}-1}{\pi_i^k-1}.
    \end{align*}
    Then, the $k$-powerful abundancy index of $\eta$ is defined as:
    \begin{align*}
        I_k(\eta) = \prod \frac{\pi_i^{k(e_i+1)}-1}{\pi^{ke_i}(\pi_i^k-1)} = \frac{u^k\sigma_k(\eta)}{\eta^k}.
    \end{align*}
\end{definition}

Observe that:

\begin{align}\label{eqn1}
    I_k(\eta) = \prod \frac{\pi^k - \frac{1}{\pi^{ke_i}}}{\pi_i^k-1} = \prod \frac{\frac{1}{\pi^{k(e_i+1)}}-1}{\frac{1}{\pi_i^k}-1} = \sigma_{-k}(\eta).
\end{align}

Finally, we also note that $\sigma_k(\eta) \in \mathbb{Z}[i] \implies I_k(\eta) \in \mathbb{Q}[i]$

We now proceed to the main theorem of the paper: to show that many of the properties of abundancy observed in integers also hold for this generalization in the ring of Gaussian integers.

\begin{theorem}[\textbf{Important properties of $I(\eta)$}]
We have the following properties of $I(\eta)$:
    \begin{enumerate}
        \item It is weakly multiplicative, i.e., for any $\alpha$ $\beta \in \mathbb{Z}[i]$ where $(\alpha, \beta) = 1$, $I_k(\alpha \beta) = I_k(\alpha)I_k(\beta)$.
        \item For $\alpha, \beta \in \mathbb{Z}[i]$ such that  $\alpha | \beta$, $|I_k(\alpha)| \leq |I_k(\beta)|$ with equality only when $\alpha \sim \beta$.
        \item For all primes $\pi \in \mathbb{Z}[i]$, $I_k(\pi^n) \neq I_k(z)$ for all possible $z \in \mathbb{Z}[i]$ and $(z, \pi^n) = 1$, i.e., all powers of primes are solitary.
        \item  For all primes $\pi \in \mathbb{Z}[i]$, $\left|\frac{\pi^k+1}{\pi^k}\right| \leq |I_k(\pi^n)| < \frac{|\pi|^k}{|\pi_i|^k-1}$
    \end{enumerate} 
\end{theorem}

\subsection{Lemmas}

Here, we provide some lemmas used in subsequent sections.

\begin{lemma}\label{ineq1}
Given a complex integer $\eta \in \mathbb{Z}[i]$ where $Re(\eta) > 0$ and $|\eta| > 1$, for an integer $n > 0$, we have the following inequality:
    \begin{align*}
        |1 + \eta + \eta^2 + \dots + \eta^n| > |\eta^n|.
    \end{align*}
\end{lemma}

\begin{proof}
The proof of the given lemma is described in detail by Spira \cite{Spira_1961}.
\end{proof}

\begin{lemma}\label{ineq2}
Given a complex integer $\eta \in \mathbb{Z}[i]$ where $Re(\eta) > 0$ and $|\eta| > 1$, for an integer $n > 0$, we have the following inequality:
    \begin{align*}
        |\eta^n-1| > |\eta^n - \eta|.
    \end{align*}
\end{lemma}

\begin{proof}
We first prove the case $n>2$ and then take the cases $n = 2$ and $n = 1$ individually.\\
We prove the given lemma for the norm of each side of the inequality which also holds true for absolute values. Consider the following expression:
    \begin{align*}
        \norm{\frac{\eta^n-1}{\eta^{n-1}-1}} = \norm{\eta - \frac{1-\eta}{\eta^{n-1} - 1}} \geq \norm{\eta} - \norm{\frac{1-\eta}{\eta^{n-1} - 1}}.
    \end{align*}

However, we observe that the left hand side is an integer and $\norm{\eta}$ is an integer. Further, making use of Lemma \ref{ineq1} and the bound $n > 2$, we observe that:

    \begin{align*}
        \norm{\frac{1-\eta}{\eta^{n-1} - 1}} &= \norm{\frac{1-\eta}{(\eta - 1)(1 + \eta  + \eta^2 + \dots + \eta^{n-2})}} \\ &= \norm{\frac{1}{1 + \eta  + \eta^2 + \dots + \eta^{n-2}}} \leq \norm{\frac{1}{\eta^{n-2}}} < 1\\
        \implies \norm{\frac{1-\eta}{\eta^{n-1} - 1}} &< 1.
    \end{align*}

Hence, we can neglect this term as it is less than 1 and the other terms are integers. Thus, the previous inequality simplifies to:

    \begin{align*}
        \norm{\frac{\eta^n-1}{\eta^{n-1}-1}} \geq \norm{\eta} \implies \norm{\eta^n-1} \geq \norm{\eta^n-\eta}.
    \end{align*}

This proves the first case of the inequality. We now show that the given inequality holds for $n=2$ and $n=1$.

For $n=2$:
    \begin{align*}
         \left|\frac{\eta^2-1}{\eta-1}\right| = |\eta + 1| > |\eta| \implies |\eta^2-1| > |\eta||\eta - 1| = |\eta^2 - \eta|,
    \end{align*}
which follows from $Re(\eta) > 0$. 

For $n=1$:
    \begin{align*}
        |\eta-1| > |\eta - \eta| = 0.
    \end{align*}

Hence, the given inequality holds for $n>1$.
\end{proof}

\begin{lemma}\label{ineq3}
Given a complex integer $\eta \in \mathbb{Z}[i]$ where $Re(\eta) > 0$ and $|\eta| > 1$, for $n \geq m \geq 0$, we have the following inequality:
    \begin{align*}
        |\eta^n-1| \geq |\eta^n - \eta^m|,
    \end{align*}
    with equality when $m=0$
\end{lemma}

\begin{proof}
    If $m = 0$, then it results in simple equality.
    
    In the case of $m>0$, the given result can be proved by repeated application of Lemma \ref{ineq2}:
    \begin{align*}
        |\eta^n - \eta^m| &= |\eta^{m-1}\|\eta^{n-m+1} - \eta| \leq |\eta^{m-1}\|\eta^{n-m+1} - 1|\\ &\leq |\eta^n - \eta^{m-1}|.
    \end{align*}
    By induction on $m$, we obtain the desired result.
\end{proof}

\section{Proofs of Properties of $I_k(\eta)$}
\begin{property}
    \label{prop1}
    For $\alpha$ $\beta \in \mathbb{Z}[i]$ where $(\alpha, \beta) = 1$, $I_k(\alpha \beta) = I_k(\alpha)I_k(\beta)$, i.e., $I_k(\alpha)$ is weakly multiplicative.
\end{property}
\begin{proof}
    From Definition \ref{def_I}, we have
    \begin{align*}
        I_k(\alpha\beta) = \frac{u^k\sigma_k(\alpha\beta)}{(\alpha\beta)^k} = u^k\frac{\sigma_k(\alpha)}{\alpha^k} \frac{\sigma_k(\beta)}{\beta^k},
    \end{align*}
    where $u$ is the unit corresponding to $\alpha\beta$. Let $v, w$ be the units corresponding $\alpha$, and $\beta$ respectively. Observe that $\alpha\beta = u\alpha' v\beta' = (uv)(\alpha'\beta')$, where $\alpha' \sim \alpha, \beta' \sim \beta$ such that $\alpha'$ and $\beta'$ are products of primes only in the first quadrant. Clearly, $\alpha'\beta'$ is a product of primes in the first quadrant and $vw$ is also a unit. Hence, we can set $u = vw$ by defenition. Thus, $u^k = v^kw^k$. Substituting:
    
    \begin{align*}
        I_k(\alpha\beta) = \frac{v^k\sigma_k(\alpha)}{\alpha^k} \frac{w^k\sigma_k(\beta)}{\beta^k} = I_k(\alpha)I_k(\beta).
    \end{align*}
    
\end{proof}

\begin{property}\label{prop2}
    For $\alpha, \beta \in \mathbb{Z}[i]$ such that  $\beta|\alpha$, $|I_k(\alpha)| \geq |I_k(\beta)|$ with equality only when $\alpha \sim \beta$.
\end{property}

\begin{proof}
    Let the prime factorization of $\alpha$ be $u\prod \pi_i^{e_i}$ and that of $\beta$ be $v\prod\rho_j^{f_j}$. Since $\beta | \alpha$, $\{\rho_j\} \subseteq \{\pi_i \}$ and for any $\pi_i = \rho_j \implies e_i \geq f_j$. 
    
    Now, we consider the following expression:
    
    \begin{align*}
        \left|\frac{I_k(\alpha)}{I_k(\beta)}\right| = \left|\left(\prod \frac{\pi_i^{k(e_i+1)}-1}{\pi_i^{ke_i}(\pi_i^k-1)}\right) \middle/ \left(\prod \frac{\rho_j^{k(f_j+1)}-1}{\rho_j^{kf_j}(\rho_j^k-1)}\right)\right|.
    \end{align*}
    
    Since we aim to prove that this expression is greater than one, we observe the terms corresponding to primes not included in $\{\rho_j\}$ increase the value of the result by a factor greater than one. Thus, they can be neglected and the equality can be replaced with a greater than or equal to. In the remaining terms, if $e_i = f_j$ for some pair of primes in the two sets, we observe that the terms in the numerator and denominator cancel. Hence, we need only consider the cases where $e_i \geq f_j$ in the set $\{\rho_j\}$ with $\pi_i = \rho_j$.
    
    Hence, our original expression reduces to the following:
    
    \begin{align*}
        \left|\frac{I_k(\alpha)}{I_k(\beta)}\right| \geq \left| \prod \frac{\rho_j^{k(e_i+1)}-1}{\rho^{k(e_i-f_j)}(\rho_i^{k(f_j+1)}-1)}\right|.
    \end{align*}
    
    However, using Lemma \ref{ineq3} with $n = k(e_i+1)$ and $m = k(e_i-f_j)$, we have the following:
    \begin{align*}
        \left| \frac{\rho_j^{k(e_i+1)}-1}{\rho^{k(e_i-f_j)}(\rho_i^{k(f_j+1)}-1)}\right| = \frac{|\rho_j^{k(e_i+1)}-1|}{|\rho_i^{k(e_i+1)}-\rho^{k(e_i-f_j)}|} \geq 1,
    \end{align*}
    with equality when $m = k(e_i-f_j) = 0\implies e_i = f_j$.
    Substituting into the previous equation:
    \begin{align*}
        \left|\frac{I_k(\alpha)}{I_k(\beta)}\right| \geq \left| \prod \frac{\rho_j^{k(e_i+1)}-1}{\rho^{k(e_i-f_j)}(\rho_i^{k(f_j+1)}-1)}\right| \geq 1 .
    \end{align*}
    with equality when $e_i = f_j$ for all the primes $\rho_j$. Thus, equality occurs only when $\alpha$ and $\beta$ have the same factorization, i.e. when $\alpha \sim \beta$.
    
    Hence, $|I_k(\alpha)| \geq |I_k(\beta)|$ with equality when $\alpha \sim \beta$.
\end{proof}

We prove the third property as a corollary of a more general theorem:
\begin{theorem}[\textbf{Greening's Criterion}]\label{greenings}
    For $\eta\in \mathbb{Z}[i]$, if $(\eta^k, \sigma_k(\eta)) = 1$ for a given integer $k$ then $\eta$ is $k$-powerful solitary, i.e., for all $z \in \mathbb{Z}[i]$ and $z \nsim \eta$, $I_k(\eta) \neq I_k(z)$ \cite{Loomis_2015}.
\end{theorem}

\begin{proof}
     We shall prove the above result by contradiction. Let there, in fact, exist $z \in \mathbb{Z}[i]$, $z \nsim \eta$, such that $I_k(\eta) = I_k(z)$. Then:
    
    \begin{align*}
        z^ku^k\sigma_k(\eta^n) = \eta^kv^k\sigma_k(z),
    \end{align*}

    where $u$ and $v$ are the units corresponding to $\eta$ and $z$ respectively. We know that $(\eta^k, \sigma_k(\eta)) = 1$. Since, $\eta^k | z^k\sigma_k(\eta)$ and $(\eta^k, \sigma_k(\eta)) = 1$, $\eta^k | z^k \implies \eta | z$ (by comparing the power of every prime $\pi$ in the factorization of $z$ and $\eta$). Since $\pi \nsim z$, we have (from Property \ref{prop2}), $I_k(\eta) < I_k(z)$. This contradicts our initial assumption. Thus, $\eta$ is $k$-powerful solitary.
\end{proof}

\begin{property}\label{prop3}
    For all primes $\pi \in \mathbb{Z}[i]$, $I_k(\pi^n) \neq I_k(z)$ for all possible $z \in \mathbb{Z}[i]$ and $z \nsim \pi^n$, i.e., all powers of primes are solitary.
\end{property}

\begin{proof}
    Clearly, $\sigma_k(\pi^n) = 1 + \pi^k + \pi^{2k} \dots + \pi^{kn} = 1 + \pi^kl$ for $l = 1 + \pi^k + \pi^{2k} \dots + \pi^{k(n-1)} \in \mathbb{Z}[i]$. Hence, $\pi^k \nmid \sigma_k(\pi^n)$. Thus, $(\pi^{kn}, \sigma_k(\pi^n)) = 1$. Hence, by Greening's Criterion (Theorem \ref{greenings}), all powers of primes are solitary.
\end{proof}

\begin{corollary}\label{corr2}
    For all primes $\pi \in \mathbb{Z}[i]$, $\pi^{n_1}\bar{\pi}^{n_2}$ is $k$-powerful solitary for all $n_1, n_2, k \in \mathbb{Z}$, i.e., there exists no value of $z \in \mathbb{Z}[i]$ such that $I_k(\pi^{n_1}\bar{\pi}^{n_2}) = I_k(z)$.
\end{corollary}

\begin{proof}
    In the case that $\pi \sim \bar{\pi}$, the above result is equivalent to Property \ref{prop3}. In the case that $\pi \nsim \bar{\pi}$, we shall prove the above result by contradiction.\\
    Let $|((\pi^{n_1}\bar{\pi}^{n_2})^k, \sigma_k(\pi^{n_1}\bar{\pi}^{n_2}))| > 1$. We observe that:
    \begin{align*}
        \sigma_k(\pi^{n_1}\bar{\pi}^{n_2}) = \sigma_k(\pi^{n_1})\sigma_k(\bar{\pi}^{n_2}).
    \end{align*}
    As proved in Property \ref{prop3}, $(\pi^{kn_1}, \sigma_k(\pi^{n_1})) = (\bar{\pi}^{kn_2}, \sigma_k(\bar{\pi}^{n_2})) = 1$. Hence, \\ $|(\bar{\pi}^{kn_2}, \sigma_k(\pi^{n_1}))| > 1$ or $|(\pi^{kn_1}, \sigma_k(\bar{\pi}^{n_2}))| > 1$. We take the first case: the other case follows in the same manner. This means that $\bar{\pi} | \sigma_k(\pi^{n_1})$. Expanding, we have:
    \begin{align*}
        \bar{\pi} | \left( 1 + \pi^k + \pi^{2k} + \dots + \pi^{kn_1} \right)
        &\implies \norm{\bar{\pi}} | \norm{1 + \pi^k + \pi^{2k} + \dots + \pi^{kn_1}} \\
        &\implies \frac{\norm{1 + \pi^k + \pi^{2k} + \dots + \pi^{kn_1}}}{\norm{\pi}} \in \mathbb{Z}\\
        &\implies \norm{\pi^{-1} + \pi^{k-1} + \pi^{2k-1} + \dots + \pi^{kn_1-1}} \in \mathbb{Z}.
    \end{align*}
    However, since $|\pi| > 1$ and $k\geq 1$, we have:
    \begin{align*}
        \norm{\pi^{-1} + \pi^{k-1} + \pi^{2k-1} + \dots + \pi^{kn_1-1}} \notin \mathbb{Z}.
    \end{align*}
    
    Hence, $|((\pi^{n_1}\bar{\pi}^{n_2})^k, \sigma_k(\pi^{n_1}\bar{\pi}^{n_2}))| = 1$. Thus, by applying Greening's Criterion (Theorem \ref{greenings}), we have the desired result.
\end{proof}

\begin{property}\label{prop4}
    For all primes $\pi \in \mathbb{Z}[i]$, $\left|\frac{\pi^k+1}{\pi^k}\right| \leq |I_k(\pi^n)| < \frac{|\pi|^k}{|\pi_i|^k-1}$
\end{property}

\begin{proof}
    From Property \ref{prop2}, we have:
    \begin{align*}
        \pi|\pi^n \implies |I_k(\pi)| \leq |I_k(\pi^n)| \implies \left|\frac{\pi^k+1}{\pi^k}\right| \leq |I_k(\pi^n)|.
    \end{align*}
    This proves the left side of the inequality. 
    
    From Equation \ref{eqn1}, we observe:
    \begin{align*}
        |I_k(\pi^n)| &= |\sigma_{-k}(\pi^n)| = \left|\sum_{i = 0}^{n} \frac{1}{\pi^{ik}}\right| \leq \sum_{i = 0}^{n} \frac{1}{|\pi|^{ik}} \\&< \sum_{i = 0}^{\infty} \frac{1}{|\pi|^{ik}} = \frac{|\pi|^k}{|\pi|^k-1}.
    \end{align*}
    This proves the right side of the inequality.
    
    The above inequality can be proved for norms as well in a similar manner.
\end{proof}

We now proceed to prove results on a direct application of the above properties: perfect numbers.

\section{Perfect Numbers} 

We begin this section with a formal definition of $k$-powerful $\tau$-perfect numbers. This is followed by a number of lemmas and, finally, a theorem in regarding the existence of such numbers for $k \in \mathbb{N}$.
\begin{definition}
    We define a complex number  $\eta \in \mathbb{Z}[i]$ to be $k$-powerful $\tau$-perfect, if, for $k \in \mathbb{N}$, $\tau \in \mathbb{Z}[i]$, and $|\tau|>1$:
    \begin{align*}
        I_k(\eta) = \tau.
    \end{align*}
    
    Along the same lines, we also define $k$-powerful $t$-norm-perfect numbers for $k \in \mathbb{N}$, $t \in \mathbb{Z}$, and $t>1$ to be complex numbers $\eta \in \mathbb{Z}[i]$ such that:
    \begin{align*}
        \norm{I_k(\eta)} = t.
    \end{align*}
\end{definition}

We now proceed to prove that such numbers exist for only $k = 1$. We first prove the following lemmas.

\begin{lemma}\label{zb1}
    For $s\in \mathbb{R}$ and $s > 1$, 
    \begin{align*}
        \zeta_{\mathbb{Q}[i]}(s) = \zeta(s)L(s, \chi_1) = \zeta(s)\beta(s),
    \end{align*}
    where $\zeta_{\mathbb{Q}[i]}(s)$ is the Dedekind Zeta function on the field $\mathbb{Q}[i]$, $\zeta(s)$ is the Riemann Zeta function, and $L(s, \chi_1) = \beta(s)$ is the Dirichlet L-function with character $\chi_1(n)$ modulo 4 (equivalent to the Dirichlet Beta function $\beta(s)$).
\end{lemma}

\begin{proof}
    Let the set of prime ideals in $\mathbb{Z}[i]$ (in the first quadrant) be represented by $\mathcal{P}$. For each prime ideal $\pi \in \mathbb{Z}[i]$, we have the following 3 cases:
    
    \begin{itemize}
        \item $\pi \sim p$ and $\norm{\pi} = p^2$ for some prime $p \in \mathbb{Z}$ such that $p \equiv 3\pmod 4$.
        \item $\pi \sim \bar{\pi}$ and $\norm{\pi} = 2$  (this case only holds for $\pi \sim 1 + i$).
        \item $\pi \nsim \bar{\pi}$ and $\norm{\pi} = \norm{\bar{\pi}} = p$ for some prime $p \in \mathbb{Z}$ such that $p \equiv 1\pmod 4$
    \end{itemize}
    
    Thus, from the definition of $\zeta_{\mathbb{Q}[i]}(s)$:
    \begin{align*}
        \zeta_{\mathbb{Q}[i]}(s) &= \prod_{\pi \in \mathcal{P}} \frac{1}{1-\norm{\pi}^{-s}} = \prod_{\substack{\pi \in \mathcal{P} \\p = \norm{\pi}}} \frac{1}{1-p^{-s}} \\ 
        &= \frac{1}{1-2^{-s}} \prod_{ p \equiv 1\mathrm{mod} 4} \frac{1}{(1-p^{-s})^2} \prod_{ p \equiv 3\mathrm{mod} 4} \frac{1}{1-p^{-2s}}\\
        &= \prod_{p} \frac{1}{1-p^{-s}} \left(\prod_{ p \equiv 1\, \mathrm{mod}\, 4} \frac{1}{1-p^{-s}} \prod_{ p \equiv 3\, \mathrm{mod}\, 4} \frac{1}{1+p^{-s}}\right)\\
        &= \zeta(s)\left(\prod_{ p \equiv 1\, \mathrm{mod}\, 4} \frac{1}{1+\chi_1(p)p^{-s}} \prod_{ p \equiv 3\, \mathrm{mod}\, 4} \frac{1}{1+\chi_1(p)p^{-s}}\right)\\
        &= \zeta(s)\prod_{ p \neq 2} \frac{1}{1+\chi_1(p)p^{-s}} = \zeta(s)L(s, \chi_1) = \zeta(s)\beta(s),
    \end{align*}
    as follows from the definitions of $\zeta(s)$, $L(s, \chi_1)$, and $\beta(s)$.
\end{proof}

\begin{lemma}\label{zb2}
    For $s\in \mathbb{R}$ and $s > 1$, $\beta(s) \leq \zeta(s)$
\end{lemma}

\begin{proof}
    By definition:
    \begin{align*}
        \beta(s) &= \prod_{ p \equiv 1\, \mathrm{mod}\, 4} \frac{1}{1-p^{-s}} \prod_{ p \equiv 3\, \mathrm{mod}\, 4} \frac{1}{1+p^{-s}} \\&\leq \frac{1}{1-2^{-s}}\prod_{ p \equiv 1\, \mathrm{mod}\, 4} \frac{1}{1-p^{-s}} \prod_{ p \equiv 3\, \mathrm{mod}\, 4} \frac{1}{1-p^{-s}} \\ &= \prod_{p} \frac{1}{1-p^{-s}} = \zeta(s).
    \end{align*}
\end{proof}

\begin{lemma}\label{zb3}
    For $\eta \in \mathbb{Z}[i]$ and $k \geq 2$:
    \begin{align*}
        \norm{I_k(\eta)} < \zeta_{\mathbb{Q}[i]}(k).
    \end{align*}
\end{lemma}

\begin{proof}
    Let $\mathcal{P}$ represent the set of prime ideals in $\mathbb{Z}[i]$ (in the first quadrant). Further, let $\mathcal{D}(\eta) = \{\pi_i\}$ be the set of primes dividing $\eta$ with $e_i = \text{ord}_{\pi_i}(\eta)$.
    Using Property \ref{prop4}, we have:
    \begin{align*}
        \norm{I_k(\eta)} &= \prod_{\pi_i \in \mathcal{D}(n)} \norm{I_k(\pi_i^{e_i})} \leq \prod  \frac{\norm{\pi_i}^k}{\norm{\pi_i}^k - 1} = \prod  \frac{\norm{\pi_i}^k}{\norm{\pi_i}^k - 1} \\
        &< \prod_{\pi \in \mathcal{P}} \frac{\norm{\pi}^k}{\norm{\pi}^k - 1} = \zeta_{\mathbb{Q}[i]}(k).
    \end{align*}
\end{proof}

Now we prove the main theorem:

\begin{theorem}
    There exist no $k$-powerful $\tau$-perfect numbers or $k$-powerful $t$-norm-perfect numbers for all $k\geq 2$.
\end{theorem}

\begin{proof}
    We take two cases: $k\geq 3$ and $k = 2$.
    
    For the first case, from Lemmas \ref{zb1}, \ref{zb2}, \ref{zb3}, we have the following:
    \begin{align*}
        \norm{I_k(\eta)} < \zeta_{\mathbb{Q}[i]}(k) = \zeta(k)\beta(k) \leq \zeta(k)^2
    \end{align*}
    Thus, since $\zeta(s)$ is monotone decreasing for $s >1$, we have:
    \begin{align*}
        \norm{I_k(\eta)} < \zeta(k)^2 \leq \zeta(3)^2 \approx 1.20206^2 < 2.
    \end{align*}
    
    For the case $k = 2$, from Lemma \ref{zb1}, we have:
    \begin{align*}
        \norm{I_2(\eta)} < \zeta_{\mathbb{Q}[i]}(2) = \zeta(2)\beta(2).
    \end{align*}
    Using $\zeta(2) = \frac{\pi^2}{6}<2$ and $\beta(2) < 1$:
    \begin{align*}
        \norm{I_2(\eta)} < \zeta(2)\beta(2) < 2.
    \end{align*}
    
    Thus, it is impossible to find any $\tau \in \mathbb{Z}[i]$ with $|\tau| \geq 2$ to satisfy $I_k(\eta) = \tau$ for any Gaussian Integer $\eta$. Hence, it is not possible for any Gaussian Integer to be $k$-powerful $\tau$-perfect for $k \geq 2$. Similarly, by the above result, the existence of $k$-powerful $t$-norm-perfect numbers is impossible in $\mathbb{Z}[i]$ for $k \geq 2$.
    
\end{proof}

\section{Conclusion}
In this paper, we have proposed a complex-to-complex definition of abundancy index for the Gaussian integers that possesses many of the properties of the traditional abundancy index in integers. We have further proved that the existence of $k$-powerful $\tau$-perfect numbers in $\mathbb{Z}[i]$ only holds for $k \geq 2$. 

We note that many of the results can be easily generalized to other imaginary quadratic rings which are unique factorization domains, i.e., we are considering rings of the form $\mathbb{Z}[\sqrt{d}]$ (if $d \equiv 1 \mod 4$) and $\mathbb{Z}[\frac{1+\sqrt{d}}{2}]$ (if $d \equiv 3 \mod 4$) for $d \in \{-2, -3, -7, -11, -19, -43, -67, -163\}$. This is because, in most of the proofs, we only make use of the fact that the norm is an integer. Since the the given rings have integral norm, it should not be very difficult to further generalize the above results to these rings. 

\section{Acknowledgments}
The author would like to thank Dr. B S Ramachandra for his continuous guidance and advice at various stages of research. The author would like to thank Dr. Maxim Gilula for all the help with writing the paper. The author would also like to thank his family for their support throughout.


\begin{thebibliography}{}
		
		\bibitem{Bouyuklieva_2013} Stefka Bouyuklieva. Applications of the {G}aussian integers in coding theory. \textit{Prospects of differential geometry and its related fields}, \textit{World Sci. Publ.}, Hackensack, NJ, 39-49, 2014.
			
		\bibitem{Defant_2015} Colin Defant. An extension of the abundancy index to certain quadratic rings. \textit{Int. J. Math. Comput. Sci.}, 9(2): 63-82, 2014.
			
		\bibitem{Hickerson_2002} Dean Hickerson. Re: friendly/solitary numbers [was: typos], unpublished correspondence, 2002.
			
		\bibitem{Loomis_2015} Paul A. Loomis. New families of solitary numbers. \textit{J. Algebra Appl.}, 14(9): 154004, 6, 2015.
			
		\bibitem{Luo_2018} David C. Luo. Generalizing the abundancy of an integer. Preprint, \url{https://arxiv.org/abs/1803.10816}.
			
		\bibitem{McDaniel_1974} Wayne McDaniel. Perfect {G}aussian integers. \textit{Acta Arith.}, 25(2): 137-144, 1974.
		
		\bibitem{Spira_1961} Robert Spira. The complex sum of divisors. \textit{Amer. Math. Monthly}, 68(2): 120-124, 1961.
		
	\end{thebibliography}
\end{document}